\documentclass[10pt]{article}

\usepackage{amsmath,amsthm,amstext,amsbsy,amssymb,amsfonts,amscd}
\usepackage{graphicx}
\usepackage{color,hyperref}
\usepackage[all]{xy}
\usepackage{mathrsfs}
\usepackage{paralist}
\usepackage[T1]{fontenc}
\usepackage{pict2e}
\usepackage{slashbox}
\usepackage{algorithm}
\usepackage{algorithmic}
\usepackage{multirow} 

\newtheorem{Th}{Theorem}[section]

\newtheorem{Conj}{Conjecture}[section]

\newcommand{\address}[1]{\vskip3mm\noindent#1}

\newcommand{\N}{{\mathbb{N}}}

\newcommand{\Z}{{\mathbb{Z}}}

\renewenvironment{matrix}[1][r]{%
  \hskip -\arraycolsep\array{*\c@MaxMatrixCols {#1}}%
}{%
  \endarray \hskip -\arraycolsep
}

\numberwithin{equation}{section}

  \newcommand{\rft}[1]{(\ref{#1})}

\begin{document}

\title{A Conjecture Connected with Units of  Quadratic Fields\\}

\author{{ Nihal Bircan}}

\maketitle
\date{}

\begin{center}
\address{ Technische Universit\"{a}t Berlin, Fakult\"{a}t II, Institut f\"{u}r Mathematics, MA $8-1$ Strasse des $17.$ Juni $136$ D-$10623$ Berlin, Germany\\
bircan@math.tu-berlin.de}
\end{center}
 
\begin{abstract}
In this article, we consider the order $\mathcal{O}_{f}=\{x+yf\sqrt{d}:x,\ y \in \Z \}$ with conductor $f\in\N$ in a
real quadratic field $K=\mathbb{Q}(\sqrt{{d}})$ where $d>0$ is square-free and $d\equiv2,3\pmod 4$. We obtain numerical information about 
$ n(f)=n(p)=min\{ \nu\in\N : \varepsilon^{\nu}\in \mathcal{O}_{p}\}$ where $\varepsilon>1$ is the fundamental unit of $K$ and $p$ is an odd prime. Our numerical results 
suggest that the frequencies of $\frac{p\pm1}{2n(p)}$ or $\frac{p\pm1}{n(p)}$ should have a limit  as the ranges of $d$ and $p$ go to infinity.

\end{abstract} 
{\bf $2010$ Mathematics Subject Classification :}{$11$Y$40$, $11$R$27$, $11$R$11$}\\
{\bf Keywords and phrases:} {Units, conductors, quadratic fields, probability \\distribution}
\section{Introduction}
Throughout the paper, we consider the real quadratic field  $K=\mathbb{Q}(\sqrt{d})$ where $d>0$ is square-free and $d\equiv2,3\pmod 4$. An {\it order} $\mathcal{O}$ in the field $K$ 
is a subset $\mathcal{O}\subset K$ such that 
\begin{compactitem}
\item[i)] $\mathcal{O}$ is a subring of $K$ containing $1$.
\item[ii)] $\mathcal{O}$ is a finitely generated $\Z$-module.
\item[iii)] $\mathcal{O}$ contains a $\mathbb{Q}$-basis of $K$.
\end{compactitem}
{\it Maximal order} of an algebraic number field $K$ is the integral closure of the ring $\Z$ of rational numbers in the field $K$. So, among various orders of the field $K$ 
there is one maximal one which contains all the other orders which we call $\mathcal{O}_{K}$. Any order $\mathcal{O}$ of the field $\mathbb{Q}(\sqrt{d})$ is of the form 
 $\mathcal{O}_{f}=\{1,f\omega\}$ where $f$ is the index $[\mathcal{O}_{K}:\mathcal{O}]$ namely, the {\it conductor} of the order. One can represent any number of the order 
$\mathcal{O}_{f}$ in the form $\mathcal{O}_{f}=\{x+yf\sqrt{d}:x,\ y \in \Z \}$. For $d\equiv2,3\pmod 4$ a basis $\{1,\omega\}=\{1,\sqrt{d}\}$. 


For $d\equiv 2,\ 3 \pmod 4$ one can write very well known Pell's equation as
\begin{equation*}
 (x+y\sqrt{d})(x-y\sqrt{d})=\pm1
\end{equation*}
%
The $s-$th solution $x_{s},y_{s}$ can be expressed in terms of the first one $x_{1},y_{1}$ by  
\begin{equation}\label{a1}
 x_{s}+y_{s}\sqrt{d}=(x_{1}+y_{1}\sqrt{d})^{s}\ \ (s\in\N)
\end{equation}
The first solution $(x_{1},y_{1})\in\N^{2}$ is called {\it fundamental solution} to the Pell equation. 
{\it Fundamental units} $\varepsilon$ for real quadratic fields $\mathbb{Q}(\sqrt {d})$ can be computed from the fundamental solution of 
Pell's equation. So, fundamental units for the maximal order $\mathcal{O}_{K}$ of the field $K=\mathbb{Q}(\sqrt{d})$ are also 
called fundamental units for the algebraic number field $K=\mathbb{Q}(\sqrt{d})$. For more information on units, orders and 
Pell's equation see \cite[p.$133$ ]{Cox}, \cite[p.$80$ ]{JW}, \cite[p.$72$  ]{Neu} and other many algebraic number theory books. For the efficient 
computation of units and solution of Pell's equation see \cite{Len}. One can also
 refer to \cite{JLW} to see about the question how large the fundamental unit $\varepsilon$ can be.


In this article, $\varepsilon$  is the fundamental unit of the real quadratic field  $\mathbb{Q}(\sqrt {d})$ and  $N(\varepsilon)=\pm1$ is its norm. We consider the case that 
the conductor $f$ is an odd prime $p$ with $p\nmid d.$ The aim of this article is to get numerical information about 
\begin{equation*}
 n(f)=n(p)=min\{ \nu\in \N : \varepsilon^{\nu}\in \mathcal{O}_{p}\}.
\end{equation*}
In \cite{BP}, we compute very good upper bounds for $n(f)$ for the case $N(\varepsilon)=+1$ and in \cite{Bir} for the case in particular that $N(\varepsilon)=-1$. 
It is known that (see for instance \cite{BP}) for  $\left(\frac{d}{p}\right)=\mp1$, $n(f)=n(p)$ is always a divisor of $\frac{p\pm1}{2}$ if $N(\varepsilon)=+1$ and of $p\pm1$ if $N(\varepsilon)=-1$. Our results become easier to state if we consider the integers, quotients $q$
 defined by
\begin{equation}\label{q1}
q= \left\{\begin{array}{lll}  \frac{p\pm1}{2n(p)}& \text{if }N(\varepsilon)=+1& \text{ where } \left(\frac{d}{p}\right)=\mp1   \\
                                   \frac {p\pm1}{n(p)} &\text{if }  N(\varepsilon)=-1& \text{ where } \left(\frac{d}{p}\right)=\mp1.
                                 \end{array}\right.
\end{equation}
We compute the frequencies of $q$. Our numerical results suggest that the frequencies may have a limit as the ranges of $p$ and $d$ go to infinity. 
We prove a theorem suggested by the numerical results and study the expectation of the frequency. 
%
\section{Computations and Conjecture}
 We recall that we always assume that $d\equiv2,3\pmod 4$ is positive, square-free and $p$ is an odd prime. We have $\mathcal{O}_{f}=\{x+yf\sqrt{d}:x,\ y \in \Z \}$. 
Every unit of $\mathcal{O}$ that belongs to $\mathcal{O}_{f}$ is also a unit of $\mathcal{O}_{f}$. 

We compute $n(f)=n(p)=min\{\nu\in \N: \varepsilon^{\nu}\in \mathcal{O}_{f}\}$
for the case that $f$ is an odd prime $p$. For $N(\varepsilon)=+1$, the possible maximal values of $n(f)$ are $\frac{p+1}{2}$ for 
$\left(\frac{d}{p}\right)=-1$, $\frac{p-1}{2}$ for $\left(\frac{d}{p}\right)=+1$; These give $q=1$ defined in \rft{q1}.

\bigskip
The computer we use has $2$x DualCore-Opteron $2218$ ($2.6$ GHz) with $16$ GB(1 Node with $32$ GB) RAM. One can also refer to the appendix to see the algorithm.

\bigskip
Our results depend on two arguments, the range and the quotient $q$. It would be difficult to give a plot. For reasons of space we only 
present the first $20$ frequencies  in our tables as percentages. We take ranges $d=[3,\dots, d_{max}]$ and $p=[3,\dots, p_{max}]$ for $N(\varepsilon)= +1$. For $N(\varepsilon)= -1$, we take $d=[2,\dots, d_{max}]$ and $p=[3,\dots, p_{max}]$.
 For $N(\varepsilon)=-1$, the computations run faster because there are few occurencies so this enables us to implement larger and different ranges. To save space and to be neat, we only write $d_{max}=p_{max}$ and some of the ranges we calculate in the following tables. The Legendre symbol differs according to the  elements $q$ defined in \rft{q1}.
For $m\in\N$ we define four sequences according to the formula \rft{q1};

$F_{1}(q;m)=\text{card}\{q=(p-1)/2n(p):d,\ p \leq m, \left(\frac{d}{p}\right)=+1,\ N(\varepsilon)=+1\}$.

$F_{2}(q;m)=\text{card}\{q=(p+1)/2n(p):d,\ p \leq m, \left(\frac{d}{p}\right)=-1,\ N(\varepsilon)=+1\}$.

$F_{3}(q;m)=\text{card}\{q=(p-1)/n(p):d,\ p \leq m, \left(\frac{d}{p}\right)=+1,\ N(\varepsilon)=-1\}$.

$F_{4}(q;m)=\text{card}\{q=(p+1)/n(p):d,\ p \leq m, \left(\frac{d}{p}\right)=-1,\ N(\varepsilon)=-1\}$.

For instance, for the case $N(\varepsilon)=+1$, by the definition \rft {q1} of $q$, we have $q\leq\frac{p-1}{2}$. Hence $F_{1}=0$ for $q> \frac{p-1}{2}$. This means that,
when the ranges become larger, more and more frequencies will become positive. This suggests that for a fixed $q$ the frequencies will slowly fluctuate. \\
Furthermore, let $S_{j}(m)=\sum_{q}F_{j}(q;m)$.  In the tables {\it values} shows the number of occurencies of $\left(\frac{d}{p}\right)=+1$ and $\left(\frac{d}{p}\right)=-1$ respectively. 
The tables suggest the following conjecture;

\begin{Conj}Let $j=1,2,3,4$. There is a probability distribution $P_{j}(q)$ such that, for all $q$,
\begin{equation*}
F_{j}(q;m)/S_{j}(m)\to P_{j}(q) \text{ as } m\to\infty.
\end{equation*}
More precisely: For $j=1,\dots,4$ there is a function
\begin{equation*}
P_{j}:\N\to [0,1]\text{ with }\sum_{q}P_{j}(q)=1
\end{equation*}
such that for every $\delta>0$ there exists $m_{0}$ with the property that
\begin{equation*}
|F_{j}(q;m)/S_{j}(m)-P_{j}(q)|<\delta
\end{equation*}
for all $q$ and $m\geq m_{0}$.
\end{Conj}

\footnotesize
\begin{center}
\begin{tabular}{|c|c|c|c|c|c|c|c|}\hline
\multicolumn{7}{|c|}{$N(\varepsilon)=+1,\left(\frac{d}{p}\right)=+1$}\\
 \hline
\backslashbox { $q$}{ $d_{max}$ }& $10000$&$20000$&$ 40000 $&$60000$&$ 80000 $&$ 100000 $\\ 
\hline
$1$& $57.3$ & $56.9$ &$56.9 $&$  56.7 $ &$ 56.7  $& $ 56.6$ \\
$2$& $11.8$& $12.0$ &$  12.0$&$ 12.1  $&$ 12.2  $& $12.2 $\\
$3$&  $9.91$&  $9.89$& $9.89  $&  $  9.88 $ &$ 9.86 $ &$9.88 $\\
$4$&  $ 4.66$& $4.72$&$ 4.64 $&  $ 4.64 $&$ 4.65 $&$4.67  $  \\ 
$5$&  $2.86$ &$2.84$ &$ 2.84 $ & $ 2.84$&$ 2.82 $&$2.82  $\\
$6$&   $1.98$& $2.05$ &$ 2.10 $ &  $2.13  $ &$2.14  $&$  2.16 $\\ 
$7$& $1.34$ & $1.35$ & $ 1.34$& $ 1.35$&$1.34   $& $ 1.34  $\\
$8$&  $1.12 $&  $1.16$  & $1.16 $&  $1.16  $ & $ 1.16 $&$ 1.16 $\\
$9$& $1.09$ &  $1.08$ & $1.09  $ &  $1.08  $& $ 1.09 $&$1.09   $ \\
$10$& $0.604$& $0.610$ &$ 0.609 $& $ 0.612 $&$ 0.618 $&$ 0.619 $ \\ 
$11$& $0.530$&$0.515$& $0.521  $&$ 0.518 $ & $0.517  $&$ 0.510  $ \\
$12$&$ 0.817$ &$0.823$ &$ 0.841 $&$ 0.854  $  &$ 0.866 $&$0.874  $\\ 
$13$&$ 0.349$ &$0.365$&$ 0.363  $ &$ 0.362 $&$0.363 $&$ 0.363  $ \\
$14$&$ 0.294$ &$ 0.292$&$ 0.289 $&$  0.292 $&$ 0.291 $ &$0.295  $  \\
$15$&$0.497 $ &$0.489$&$  0.494 $&$ 0.492  $&$ 0.491$ &$ 0.495$\\
$ 16$&$0.280 $ &$0.306 $&$ 0.302  $ &$0.296  $&$0.293 $&$  0.293$\\
$ 17$&$ 0.216$ &$0.214 $&$0.210  $ &$ 0.214 $&$ 0.213$&$0.210 $\\
$ 18$&$ 0.221$ &$0.229 $&$ 0.237 $&$ 0.239 $&$ 0.238$&$  0.240$ \\
$ 19$&$0.161  $ &$0.164 $&$0.159  $ &$ 0.164  $&$ 0.163$ &$ 0.162$\\
$ 20$&$ 0.233  $ &$ 0.238$&$  0.239 $&$  0.236$ &$ 0.238$&$0.240 $ \\
$\vdots$&$ \vdots$ &$\vdots$&$ \vdots$&$ \vdots$ &$ \vdots$&$ \vdots$\\ \hline
Values &  $2249621$ &$8330759$   &   $31140582$ &$67496484 $& $ 116576513 $& $ 178609439 $\\ \hline
\end{tabular}
\end{center}

\footnotesize
\begin{center}
\begin{tabular}{|c|c|c|c|c|c|c|c|}\hline

\multicolumn{7}{|c|}{$N(\varepsilon)=+1, \left(\frac{d}{p}\right)=-1$}\\ \hline
\backslashbox { $q$}{ $d_{max}$ }& $10000$& $20000$& $40000$&$60000$&$ 80000 $& $ 100000 $\\ \hline
$1$& $56.3$&$56.2$ &$ 56.0 $&$ 56.1  $&$ 56.1  $ & $ 56.1$\\
$2$& $14.5 $&  $14.4$& $ 14.4 $ &  $ 14.3 $ & $14.3  $&$ 14.4$\\
$3$&$9.94$&$9.97$& $9.98  $&  $ 9.97 $ &$ 10.0  $&$ 9.98$ \\
$4$& $ 3.32$ & $3.31$ &$3.36  $&  $ 3.37 $ &$3.37  $& $ 3.34  $ \\
$5$& $2.77 $& $2.76$&$2.80  $& $2.80 $ &$2.80   $&$  2.81 $ \\
$6$& $2.45$ & $2.48$&$2.49  $&  $2.49  $ &$ 2.49 $&$  2.50$  \\
$7$&$ 1.32 $& $1.36$& $1.35 $&$1.36 $&$1.36 $ & $ 1.35  $\\
$8$&  $0.868$&  $0.840$& $ 0.849$&  $0.836$& $0.839  $& $ 0.837 $\\
$9$& $1.07$ &$1.10$& $  1.10$&  $1.11  $ &$1.12  $&$ 1.12  $\\
$10$&$0.718$ &$0.728$&$ 0.721 $&$0.725  $ & $0.736  $&$0.729   $ \\ 
$11$& $ 0.512$&$0.502$& $ 0.518$&$0.518  $ & $0.519  $&$0.514  $\\
$12$&$ 0.561$ &$0.566$&$0.583  $&$0.581  $ &$0.578 $&$0.571   $\\ 
$13$&$ 0.352$ &$0.363$&$ 0.352 $&$0.355  $ &$ 0.351 $&$0.354  $\\
$14$&$0.341  $ &$ 0.339$&$0.328  $&$0.339  $ &$0.333 $&$ 0.332   $ \\
$15$&$0.480 $ &$0.487$&$ 0.508  $&$ 0.507 $ &$0.508 $&$0.506 $ \\
$16 $&$ 0.202$&$0.200 $ &$ 0.206  $ &$ 0.206$&$ 0.210 $&$ 0.209$  \\
$17 $&$ 0.190$&$ 0.199$ &$ 0.201 $ &$0.201 $&$0.200  $&$0.204 $  \\
$18 $&$0.274 $&$0.268 $ &$0.274  $ &$0.274 $&$0.272  $&$ 0.273$  \\
$ 19$&$ 0.165 $&$0.168 $ &$0.169  $ &$ 0.168$&$ 0.168 $&$ 0.165$  \\
$20 $&$0.177 $&$  0.176$ &$  0.175 $ &$ 0.175$&$0.174  $&$0.171 $ \\
$\vdots$&$ \vdots$ &$ \vdots$ &$\vdots$&$\vdots$&$\vdots$&$\vdots$\\\hline
 Values &  $2272057$ &$8390244$&  $ 31303879  $&  $67789746  $&  $ 117013651 $& $ 179198341 $\\ \hline
\end{tabular}
\end{center}

\footnotesize
\begin{center}
\begin{tabular}{|c|c|c|c|c|c|c|c|c|}\hline
\multicolumn{7}{|c|}{$N(\varepsilon)=-1, \left(\frac{d}{p}\right)=+1$}\\ \hline
\backslashbox { $q$}{ $d_{max}$ }& $40000 $ &$60000$& $80000$&$ 100000  $& $ 200000 $& $400000$ \\
\hline
$1$& $ 37.8  $ &$ 37.6 $&$37.6   $&$ 37.5  $& $ 37.4 $& $37.5 $ \\
$2$& $18.7  $  & $18.8  $&  $18.8  $& $ 18.8 $&$18.7  $ &$18.7 $ \\
$3$& $6.58 $& $  6.56$&  $ 6.60 $& $6.62  $&$  6.64 $ &$6.62 $ \\
$4$& $14.0 $ & $14.0   $&  $ 14.0 $&$14.0  $& $14.0  $&$14.0  $ \\ 
$5$& $ 1.90 $ & $ 1.91 $& $ 1.89 $&$1.90  $&$ 1.89  $&$1.90  $\\
$6$& $3.28 $  & $ 3.30  $&  $ 3.29 $&$3.29  $&$3.31  $ &$ 3.31  $ \\ 
$7$&$ 0.891$& $0.897 $&  $ 0.899 $&$ 0.898  $& $0.894  $ & $0.892  $\\
$8$& $3.46  $ &  $3.48    $&  $ 3.49 $& $3.51  $& $3.50  $&$3.49  $\\
$9$& $ 0.724$  &$  0.716 $&  $0.725  $ & $ 0.728 $&$0.733  $&$ 0.735   $ \\
$10$& $0.930 $ & $0.928   $&$ 0.936 $& $ 0.937 $&$0.943  $&$0.938   $ \\ 
$11$& $0.352 $&$0.351  $&$0.348$&$ 0.340 $&$ 0.343 $ &$0.339  $ \\ 
$12$& $2.49 $ &$2.50  $&$2.48  $&$2.48  $&$2.49    $&$ 2.49 $  \\
$13$& $0.237 $&$0.239 $&$0.240  $&$ 0.243   $&$0.239  $ &$0.245  $ \\
$14$& $ 0.440$&$0.440   $&$ 0.438 $&$0.442   $&$0.444 $ &$0.447$ \\
$15$& $ 0.329$&$ 0.329 $&$  0.330 $&$ 0.333  $&$ 0.333 $ &$0.334 $ \\
$16 $&$ 0.866$&$0.868 $ &$ 0.863  $ &$0.866 $&$0.872  $&$0.874 $ \\
$ 17$&$0.140 $&$0.143 $ &$0.143  $ &$0.140 $&$0.138  $&$0.140 $ \\
$18 $&$0.371 $&$0.371 $ &$ 0.365 $ &$0.366 $&$0.369  $&$ 0.368$  \\
$19 $&$0.105 $&$0.109 $ &$ 0.108 $ &$ 0.107 $&$ 0.109 $&$0.109 $ \\
$ 20$&$ 0.705$&$ 0.706$ &$0.705  $ &$0.702 $&$ 0.706 $&$0.703 $  \\
$\vdots$&$\vdots $&$ \vdots$ &$ \vdots$ &$\vdots$&$\vdots $&$ \vdots$ \\\hline
 Values &$ 2812857$&$5965306 $ &  $ 10163034$&$15404441  $&$56042335 $ &$ 205444859    $   \\ \hline
\end{tabular}
\end{center} 
\footnotesize
\footnotesize
\begin{center}
\begin{tabular}{|c|c|c|c|c|c|c|c|}\hline 
\multicolumn{7}{|c|}{$N(\varepsilon)=-1,\ \left(\frac{d}{p}\right) =-1$}\\ \hline
\backslashbox { $q$}{ $d_{max}$ }& $40000 $&$60000$ & $80000$ &$ 100000  $ & $ 200000 $& $400000$ \\
\hline
$1$& $37.9 $&$ 37.6 $& $ 37.6  $&   $37.6   $ & $ 37.6 $& $37.5 $ \\
$2$&$37.1 $& $37.2  $&  $37.2  $ & $37.2  $&$ 37.3  $ &$ 37.3$ \\
$3$&$  6.66$& $ 6.65 $&  $6.64   $ & $6.62  $&$ 6.61  $ &$6.65 $ \\
$4$&$0.000 $& $0.000 $& $0.000   $& $ 0.000 $& $0.000 $&$ 0.000 $ \\ 
$5$& $1.90 $&$ 1.90  $& $1.92 $ &$1.90  $&$ 1.90  $&$1.89  $\\
$6$& $6.59 $&$ 6.60  $&   $ 6.64  $ &$6.63  $&$  6.64$ &$6.64   $ \\ 
$7$&  $ 0.889 $&$0.910 $&  $0.897  $ &$  0.891  $& $ 0.891 $ & $ 0.899 $\\
$8$& $0.000 $&$0.000 $& $ 0.000   $&  $0.000  $& $0.000  $&$0.000  $\\
$9$&$ 0.727 $&$0.732  $&   $ 0.729 $ & $0.723  $&$ 0.731 $&$0.733   $ \\
$10$& $1.85 $& $ 1.85   $&$1.84  $ & $ 1.87 $&$ 1.87 $&$ 1.88   $ \\
$11$&$0.344 $&$ 0.345  $&$0.347  $ &$0.351  $&$0.338  $ &$ 0.343 $ \\ 
$12$&$0.000 $&$0.000  $&$  0.000 $ &$0.000 $&$0.000 $ &$ 0.000  $ \\
$13$&$ 0.233 $&$0.236 $&$0.238  $ &$0.240  $&$ 0.241 $ &$0.239  $ \\
$14$& $0.895 $&$ 0.907  $&$0.907  $ &$0.900 $&$0.899 $ &$ 0.895   $ \\
$15$& $0.342 $&$ 0.336  $&$0.336  $ &$0.332 $&$0.334 $ &$ 0.333   $ \\
$16 $&$ 0.000 $&$ 0.000 $ &$ 0.000  $ &$ 0.000 $&$ 0.000  $&$ 0.000 $ \\
$17$&$ 0.135  $&$0.138 $ &$0.139  $ &$0.140 $&$ 0.138  $&$0.138 $ \\
$18 $&$0.734 $&$0.741 $ &$ 0.750 $ &$0.754 $&$0.742  $&$0.742 $  \\
$19 $&$0.112 $&$ 0.112$ &$ 0.113 $ &$ 0.111 $&$ 0.110 $&$0.110 $ \\
$ 20$&$0.000 $&$0.000 $ &$ 0.000 $ &$0.000 $&$0.000  $&$0.000 $  \\
$\vdots$&$\vdots $&$ \vdots$ &$ \vdots$ &$\vdots$&$\vdots $&$ \vdots$  \\ \hline
 Values & $2828439 $&$5992331$&  $10206629$ &$15464072   $ & $56197833   $& $205855014  $\\ \hline
 \end{tabular}
 \end{center}
\newpage
\normalsize
{\it Comments:} The first two tables appear to be rather similar. In the first two tables $q=1$ predominates. In the third table the sum of frequencies for $q=1$ and $q=2$
is about the frequency for $q=1$ of the first two tables.

But the last table is quite different. The first two frequencies are almost equal and    
%
we notice that when $q=4k$ for $k\in\N$, $F_{4}(4;m)=0$. Now, we give a proof that this is indeed true. 
\begin{Th}
 Let $N(\varepsilon)=-1,\ \left(\frac{d}{p}\right)=-1$. Then the case $q=4$ does not appear.
\end{Th}
\begin{proof} Suppose that $q=4$ does occur. Then $F_{4}(4;m)\neq0$ and thus $\nu=\frac{p+1}{4}.$ In \cite{Bir} we wrote
\begin{equation*}
 \varepsilon^{n}=A_{\nu}+B_{\nu-1}v\sqrt{d}
\end{equation*}
In our case $\nu=\frac{p+1}{4}$ and thus $B_{\frac{p+1}{4}-1}$. Since $\nu\in\N$ it follows that $p\equiv 3\pmod 4.$
In \cite{Bir} Theorem $3.1$ we proved that, for  $\left(\frac{d}{p}\right)=-1$, $A_{\frac{p+1}{2}}\equiv 0 \pmod p$ if $\left(\frac{N}{p}\right)=-1$ and  for $N\in\Z$, $N\neq0$
by using the following formula 
\begin{equation*}
2\left(x^{2}-N\right)B_{\nu-1}(x)^{2}=A_{2\nu}(x)-N^{\nu}
 \end{equation*}
we obtain
\begin{align*}
 2\left(x^{2}+1\right)B_{\frac{p+1}{4}-1}(x)^{2}&\equiv A_{\frac{p+1}{2}}(x)-(-1)^{\frac{p+1}{4}} \\
&\equiv1 \pmod p
\end{align*}
So, $B_{\frac{p+1}{4}-1}\not\equiv 0\pmod p$ which is a contradiction.
\end{proof}
We compute the expectation of the probability distribution for $N(\varepsilon)=+1$, $\left(\frac{d}{p}\right)=+1$ which are listed below.
\footnotesize
\begin{center}
\begin{tabular}{|c|c|c|c|c|c|c|c|c|}\hline
{\it Ranges$=m$}&$1000$& $10000$&$15000$&$20000$&$ 25000 $& $ 30000 $ \\ 
\hline
{\it Expectation$=E$}&$ 3.921$& $6.086 $ &$  6.397$&  $ 6.759$ &$6.946  $ & $ 7.098 $ \\\hline
$\log(m)\times\log(\log m) $&$13.344 $&$20.448 $&$21.758 $&$22.697 $&$23.441 $&$24.047 $\\ \hline
$E/(\log(m)\times\log(\log m)) $&$0.293$&$0.297 $&$0.294 $&$ 0.297$&$0.296 $&$0.295 $\\ \hline
\end{tabular}
\end{center}
\begin{center}
\begin{tabular}{|c|c|c|c|c|c|c|c|}\hline
{\it Ranges$=m$}&$ 40000$& $50000$ &$ 60000  $&  $ 70000$ &$ 80000  $ & $ 100000$ \\ 
\hline
{\it Expectation$=E$}&$ 7.300$& $ 7.472$ &$ 7.681  $&  $7.785$ &$7.922 $ &  $8.152 $\\ \hline
$\log(m)\times\log(\log m) $&$25.012 $&$ 25.763$&$26.383 $&$26.536 $&$27.362 $&$28.129 $\\ \hline
$E/(\log(m)\times\log(\log m) )$&$ 0.291$&$0.290 $&$0.291 $&$0.293 $&$0.289 $&$0.289 $\\ \hline
\end{tabular}
\end{center}
\normalsize
This table suggests that the expectation has the order of $\log(m)\times\log(\log m)$ for large $m$ where $m$ is the range of $d$ and $p$.
\setcounter{secnumdepth}{0}
\appendix
\section{Appendix}

We use the following procedure to obtain the values. The computation is similar for $N(\varepsilon)= -1.$ At each step of the computation the numbers are modulo $p$ so that they are 
$32$ bit numbers. To speed things up, one should not calculate $\varepsilon^{i}$ as an element of the number field. It is faster to just calculate the coefficients modulo $p$ 
with respect to the basis $\{1,\sqrt{d}\}$. They can be calculated with a fast exponentiation algorithm. Computations are made by using \cite{BCP}.

\begin{algorithm}
 \caption{Finding $n(f)=n(p)$}
\
\begin{algorithmic}[1]
\bigskip
\STATE Compute the Norm of $\varepsilon$ and the Legendre Symbol $ls\leftarrow \left(\frac{d}{p}\right)$
\IF {\text{Norm}$(\varepsilon)= 1$}
\STATE {$q \leftarrow (p-ls)/2$.}
\ELSE
\STATE $q \leftarrow p-ls$
\ENDIF
\STATE Calculate factorization of $q=p_{1}^{r_{1}}\dots p_{t}^{r_{t}}$ and set $\nu\leftarrow 1 $
\FOR {$i$ in $[1..t]$} 
\STATE Calculate $F[i]\leftarrow  \text{max}\{b\in \{0,\dots,r_{i}\}|$\text{ for which the second coefficient}
$ y_{{q}/{p_{i}^{b}}}$ \text{(see \rft{a1}) } \text{of}  \  $ \varepsilon^{\frac{q}{p_{i}^{b}}}$ is divisible by $p$\} \\
$\nu \leftarrow \nu\cdot p_{i}^{F[i]} $
\ENDFOR
\STATE Compute $\frac{p\pm1}{2n(p)}$ for $N(\varepsilon)=1$ or $\frac{p\pm1}{n(p)}$ for $N(\varepsilon)=-1.$
\end{algorithmic}
\end{algorithm}

\begin{center}
\bf{Acknowledgements}
\end{center}
$\bullet$ I would like to thank Prof. Dr. Michael E. Pohst for being my PhD supervisor and for his guidance during this work.

$\bullet$ I would like to thank Prof. Dr. Franz Halter-Koch for suggesting this very interesting problem.

$\bullet$ I would like to thank Gerriet M\"{o}hlmann very much for helping me to implement and repeatedly improve the algorithm.

{\small
\def\refname{References}
\newcommand{\etalchar}[1]{$^{#1}$}

\end{document}